\documentclass{amsart}
\usepackage{amssymb}
\usepackage{fullpage}
\usepackage{tikz}
\usetikzlibrary{arrows.meta,decorations.markings}
\usepackage{subcaption}
\usepackage{longtable}
\usepackage{url}
\definecolor{mone}{HTML}{0F8B8D}
\definecolor{mtwo}{HTML}{E07A1F}
\definecolor{pcol}{HTML}{C0392B}
\definecolor{hcol}{HTML}{1F4E9C}
\definecolor{fcol}{HTML}{55606E}
\definecolor{twcol}{HTML}{7B3FA0}
\tikzset{
  midarrow/.style={postaction={decorate}, decoration={markings, mark=at position .6 with {\arrow{Stealth[length=#1]}}}},
  midarrow/.default=3.8pt,
  bv/.style={circle, fill=black, inner sep=0pt, minimum size=3.4pt},
  bvn/.style={circle, fill=black!30, inner sep=0pt, minimum size=3.4pt},
  wv/.style={circle, draw=black, fill=white, line width=.5pt, inner sep=0pt, minimum size=3.4pt},
  wvf/.style={circle, draw=black!30, fill=white, line width=.4pt, inner sep=0pt, minimum size=2.6pt},
  wvd/.style={circle, draw=twcol, fill=white, line width=.8pt, inner sep=0pt, minimum size=3.6pt},
  fv/.style={rectangle, fill=fcol, inner sep=0pt, minimum size=3pt},
  fvn/.style={rectangle, fill=black!30, inner sep=0pt, minimum size=3pt},
  gE/.style={line width=.6pt},
  gEf/.style={black!22, line width=.4pt},
  eE/.style={line width=1.8pt},
  fbE/.style={fcol!80, line width=.4pt},
  fwE/.style={fcol!40, line width=.35pt},
  nbE/.style={black!28, line width=.4pt, dash pattern=on 1.5pt off 1pt},
  hE/.style={black!60, line width=.45pt},
  m1E/.style={mone, line width=1.3pt},
  m2E/.style={mtwo, line width=1.3pt, dash pattern=on 2.5pt off 1.2pt},
  m2f/.style={mtwo, line width=.9pt, dash pattern=on 1pt off 1pt},
  tE/.style={black, line width=.55pt},
  dE/.style={twcol, line width=1.2pt},
  oE/.style={black, line width=.7pt, midarrow},
  pE/.style={pcol, line width=1.4pt, midarrow},
  pE0/.style={pcol, line width=1.4pt},
  pF/.style={pcol!70, line width=.6pt, dash pattern=on 1.5pt off 1pt},
  hamE/.style={hcol, line width=1.6pt},
  m1A/.style={mone, line width=1.2pt, midarrow=5.5pt},
  m2A/.style={mtwo, line width=1.2pt, dash pattern=on 3pt off 1.5pt, midarrow=5.5pt},
  rayE/.style={black!65, line width=.6pt, dash pattern=on .8pt off 1.2pt, -{Stealth[length=4pt]}},
  xing/.style={circle, draw=black, fill=white, line width=.5pt, inner sep=0pt, minimum size=3.2pt},
  dcurve/.style={black!45, line width=.8pt, dash pattern=on 1pt off 1.5pt},
  insideD/.style={fill=black!7},
  thirdE/.style={black!55, line width=1.2pt},
  rotA/.style={black, line width=.6pt, -{Stealth[length=3.5pt]}},
  lbl/.style={font=\scriptsize, inner sep=1pt},
  note/.style={font=\footnotesize, inner sep=1pt},
}

\newtheorem{theorem}{Theorem}
\newtheorem{lemma}[theorem]{Lemma}
\newtheorem{claim}[theorem]{Claim}
\newtheorem{corollary}[theorem]{Corollary}
\theoremstyle{definition}
\newtheorem{definition}[theorem]{Definition}
\newtheorem{remark}[theorem]{Remark}

\newcommand{\restatedref}{}
\newcommand{\restatedname}{}
\theoremstyle{plain}
\newtheorem*{restated}{\restatedname~\restatedref}
\newenvironment{restatement}[2][Lemma]
  {\renewcommand{\restatedname}{#1}\renewcommand{\restatedref}{\ref{#2}}\begin{restated}}
  {\end{restated}}

\newcommand{\cO}{\mathcal{O}}
\newcommand{\GFB}{G_{FB}}

\newcommand{\symd}{\mathbin{\triangle}}
\newcommand{\cw}{\mathrm{cw}}
\newcommand{\dir}{\mathrm{dir}}
\newcommand{\sgn}{\mathrm{sign}}
\newcommand{\orient}{\mathrm{orient}}
\newcommand{\aux}{\mathrm{aux}}
\newcommand{\trB}{\mathrm{trace}_B}
\newcommand{\Pbar}{\bar{P}}
\newcommand{\Int}{\mathrm{Int}}
\newcommand{\Ext}{\mathrm{Ext}}

\title{A note on the proof of Barnette's conjecture}
\author{Zhentao Li}
\thanks{Email: arxiv@zli2.com}

\begin{document}

\begin{abstract}
Barnette's conjecture states that every cubic 3-connected bipartite planar graph has a Hamiltonian cycle. We present a digested exposition of OpenAI's generated proof of Barnette's conjecture \cite{openai-barnette}.

\end{abstract}

\maketitle
\section{Introduction}

In this note, we provide an exposition of OpenAI's recently generated proof of Barnette's conjecture \cite{openai-barnette} that may be more accessible.

\begin{definition}
A \emph{Barnette graph} is a cubic 3-connected bipartite planar graph.
\end{definition}

\begin{theorem}[Barnette's conjecture \cite{barnette}]
Every Barnette graph has a Hamiltonian cycle.
\end{theorem}

The proof presented here is essentially unchanged, but replaces some parts of the generated proof with combinatorial arguments. In particular, we replace the complex-valued exponential sum with direct bijections (namely, two sign-reversing involutions). This note focuses solely on the proof.

We'll prove a stronger version of the statement.

\begin{theorem}\label{stronger}
For every Barnette graph $G$ and every edge $e_0 \in G$, $G - e_0$ has a Hamiltonian cycle.
\end{theorem}

For the remainder of this note, we'll also assume $G$ has no 3-edge cut with more than one vertex on each side. Graphs without such cuts are called \emph{cyclically 4-edge-connected}. We leave the details of the lower connectivity case to Appendix~\ref{sec:3edgescut}. So we'll prove the following.

\begin{theorem}\label{cyclic4}
For every cyclically 4-edge-connected Barnette graph $G$ and every edge $e_0 \in G$, $G - e_0$ has a Hamiltonian cycle.
\end{theorem}

We now sketch the proof. In the first step, we build the face-black vertex incidence graph of $G$.

\begin{definition}
(See Figure~\ref{fig:gfb}.) The face-black vertex incidence graph of a planar bipartite graph $G$ with bipartition $(B, W)$, written $\GFB$, is defined as
\begin{itemize}
\item $V(\GFB) = B \cup F(G)$
\item $E(\GFB) = \{bf : b \in f\}$
\end{itemize}
\end{definition}

\begin{figure}[t]
\centering
\input{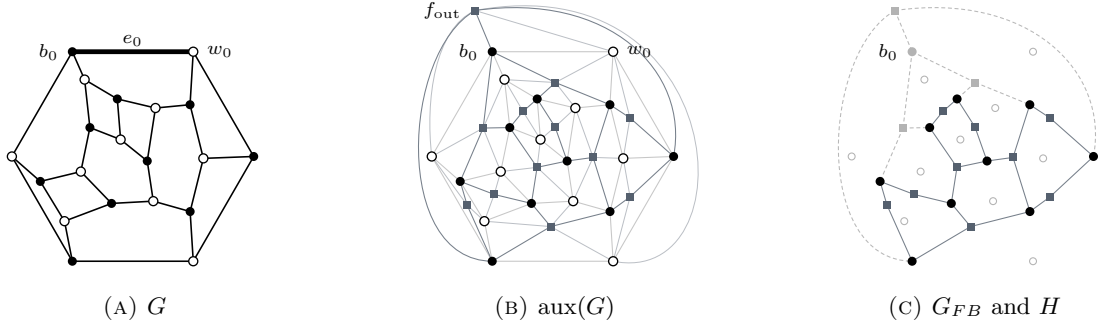}
\caption{(a) A Barnette graph $G$ and edge $e_0 = b_0w_0$ on the outer face. (b) $\aux(G)$ (c) $\GFB$ (gray and black) and $H$ (black only).}
\label{fig:gfb}
\end{figure}

\begin{proof}[Proof sketch]
See Figure~\ref{fig:sketch}. Let $(B, W)$ be the bipartition of $G$ with vertices coloured black and white, and $e_0 = b_0w_0$ with $b_0$ black and $w_0$ white.

Fix a drawing of $G$ where $b_0w_0$ is on the outer face and vertices are in general position. We find the desired Hamiltonian cycle in $G$ in three steps.

First we find a perfect matching $M$ of $\GFB - N[b_0]$ such that $\GFB - M$ is a spanning tree $T_{FB}$ of $\GFB$.

White vertices can be thought of as faces of $\GFB$. The planar dual of the complement of $T_{FB}$ is a spanning tree of the white vertices $T_W$. (This is the tree-cotree decomposition.)

Second, we turn $T_W$ into a spanning tree $T_{BW}$ of $G$ by ``tracing around it'', replacing each dual edge by a length two path through the black vertex of the corresponding primal edge.

Finally, we orient $T_{BW}$ towards $w_0$ and delete $e_0$ and the matching consisting of all outgoing edges from white vertices in this orientation. We claim the 2-factor that remains is a Hamiltonian cycle.
\end{proof}

\begin{figure}[t]
\centering
\input{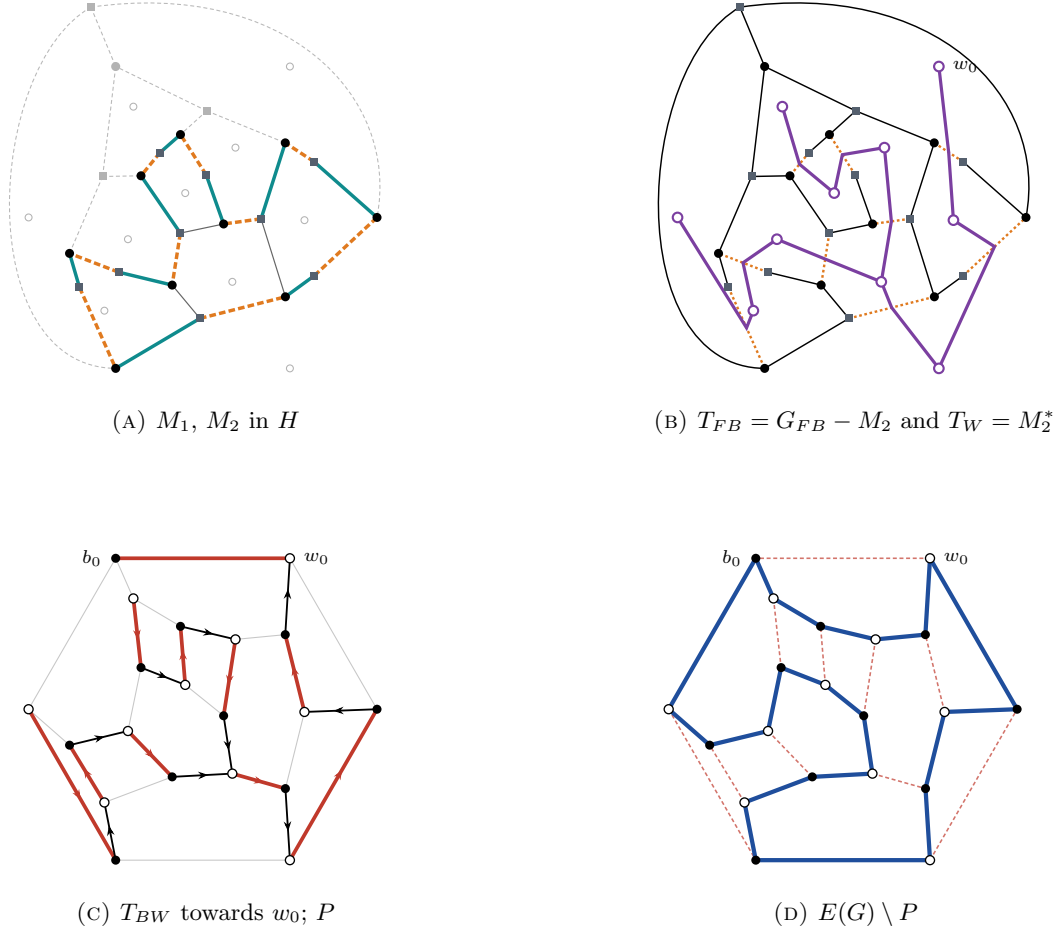}
\caption{Graphs in the proof sketch (a) edge-disjoint perfect matchings $M_1$, $M_2$ of $H$ (an ``ordered 2-factor'') (b) $T_{FB} = \GFB - M_2$ (black) and $T_W = M_2^*$ (purple) (c) ``trace'' of $T_W$ oriented towards $w_0$ and matching $P$ (d) $E(G) \setminus P$}
\label{fig:sketch}
\end{figure}

In what remains, we will show the cycle-breaking perfect matching and thus the spanning trees $T_{FB}$ exist, that all transformations described above are well defined and produce a Hamiltonian cycle of $G$.

In Section~\ref{sec:involutions}, we prove the Key Lemma showing the existence of the matching.

\begin{definition}
A \emph{rooted Barnette graph} $(G, B, W, b_0w_0)$ is a Barnette graph $G$ with bipartition $B$, $W$ together with an edge $b_0w_0 \in E(G)$, $b_0 \in B$.

An \emph{ordered 2-factor} $(M_1, M_2)$ of a graph $H$ is an ordered pair of edge-disjoint perfect matchings $M_1$ and $M_2$ of $H$.

A \emph{cycle-breaking 2-factor} of $H$ is an ordered 2-factor $(M_1, M_2)$ of $H$ such that $H - M_2$ is acyclic.
\end{definition}

\begin{lemma}[Key Lemma]\label{lem:matching}
Let $(G, B, W, b_0w_0)$ be a rooted Barnette graph with $G$ cyclically 4-edge-connected. Then $H = \GFB - N[b_0]$ has a cycle-breaking 2-factor.
\end{lemma}

In Section~\ref{sec:spanning-tree} we prove the complement $\GFB - M_2$ of the matching is indeed a spanning tree $T_{FB}$.

\begin{lemma}\label{lem:spanning-tree}
Let $(G, B, W, b_0w_0)$ be a rooted Barnette graph. If $H = \GFB - N[b_0]$ has a cycle-breaking 2-factor $(M_1, M_2)$ then $\GFB - M_2$ is a spanning tree of $\GFB$.
\end{lemma}

In Section~\ref{sec:ham}, we formally define the ``black dual trace'' and orientation of the edges of the black dual trace of $T_W$. We then prove that removing $e_0$ and all outgoing edges from white vertices in the black dual trace of $T_W$ results in a Hamiltonian cycle of $G$.

\subsection{Planar embedding}

We close this introduction by fixing a planar embedding that we use throughout.

Another way to construct the face-black vertex incidence graph is to take $G$, add a vertex inside each face connected to each black vertex on that face and delete all white vertices. We call $\aux(G)$ the graph obtained from $G$ by adding a vertex inside each face, connected to all black and white vertices on that face.

$\aux(G)$ is tripartite with tripartition $(B, W, F(G))$ and the subgraph of $\aux(G)$ induced by two parts has exactly one vertex of the third part inside its face. So we may think of $W$ as faces of $\GFB$.

For the rest of this note, we fix a straight-line embedding of $\aux(G)$ where $e_0$ is on the outer face and vertices are in general position. All other graphs are viewed as subgraphs of $\aux(G)$ and keep this drawing.

\section{Basic definitions}

This section can be skipped and referenced when needed.

Let $G$ be planar and fix a drawing of $G$ where vertices are in general position.

A cycle $C$ of $G$ has a corresponding curve in the plane, separating the plane into two regions. Call one the \emph{interior} $\Int(C)$ and the other the \emph{exterior} $\Ext(C)$ so $\Ext(C)$ contains the outer face. Both include the boundary $C$. We abuse notation and also write $\Int(C)$ to mean the subgraph of $G$ induced by $\Int(C)$. We say that a vertex $v$ is \emph{inside} $C$ if $v \in \Int(C)$.

We can similarly define $\Ext(C_1, \ldots, C_k)$ for disjoint cycles $C_1$, ..., $C_k$ to mean the region outside all cycles.

A \emph{facial cycle} of $G$ is a cycle made of all edges of some face of $G$.

The planar dual of $G$ is denoted $G^*$ and is the graph whose vertices are $F(G)$ and two faces (as vertices in $G^*$) are adjacent if they share an edge in $G$.

$N(v)$ is the set of neighbours of $v$, $N[v] = N(v) \cup \{v\}$.

A \emph{2-factor} of a graph $G$ is a 2-regular spanning subgraph of $G$.

For a subset of vertices $X$ of $V(G)$, $\delta(X)$ is the set of all edges from $X$ to $V(G) \setminus X$.

A \emph{$k$-edge cut} is a $\delta(X)$ for some $X$ such that $|\delta(X)| = k$ and $X$ and $V(G) - X$ are non-empty. A \emph{non-trivial} $k$-edge cut is a $k$-edge cut where both $X$ and $V(G) - X$ have at least two vertices.

For an integer $x \geq 0$, $x \oplus 1$ is $x$ with its lowest binary bit flipped. i.e., it swaps 0 and 1, 2 and 3, 4 and 5, and so on.

An involution is a function $f$ where $f(f(x)) = x$ everywhere. i.e., it is the identity when applied twice. An element $x$ such that $f(x) = x$ is a \emph{fixed point} of $f$ and we say $f$ \emph{fixes} $x$.

\section{Finding a 2-factor}

As a first step to proving the Key Lemma, we show $H$ has a 2-factor, without requiring (for the moment) a perfect matching of the 2-factor to hit all cycles.

\begin{lemma}\label{two-factor}
Let $(G, B, W, b_0w_0)$ be a rooted Barnette graph with $G$ cyclically 4-edge-connected. Then $H = \GFB - N[b_0]$ has a 2-factor.
\end{lemma}

Applying Ore's theorem to $H$ is straightforward. We use the following form.

\begin{theorem}[\cite{ore}]\label{ore}
Let $H$ be a bipartite graph with bipartition $(B', F')$. $H$ has a 2-factor exactly when $\sum_{b \in B'} \min(2, |N(b) \cap S|) \geq 2|S|$ for every $S \subseteq F'$.
\end{theorem}

So it remains to prove this condition holds for $H$.

\begin{definition}
For $S \subseteq F' = F(G)-N(b_0)$, let $d(S) = 2|S| - \sum_{b \in B'} \min(2, |N(b) \cap S|)$ be the \emph{deficit} of $S$.
\end{definition}

Let $B_{S, i} = \{b \in B \mid |N_H(b) \cap S| = i\}$ and $W_{S, i} = \{w \in W \mid |N_{\aux(G)}(w) \cap S| = i\}$. We omit the index $S$ when it is clear from context.

Rewrite $d(S)$ as $d(S) = 2|S| - (|B_{S, 1}| + 2|B_{S, 2}| + 2|B_{S, 3}|)$

\begin{proof}[Proof of Lemma~\ref{two-factor}]
Suppose for a contradiction, there's a set with positive deficit. Let $S \subseteq F(G)-N(b_0)$ be a minimal set with positive deficit.

Let $G_S$ be the subgraph of $G$ (and $\aux(G)$) where

\begin{itemize}
\item $E(G_S)$ is the union of edges on all facial cycles of faces in $S$ (face-vertex of $S$ seen as faces of $G$) and
\item $V(G_S)$ is all (black and white) vertices that are ends of edges in $E(G_S)$.
\end{itemize}

We view $G_S$ as a subgraph of $G$ and will refer to its faces rather than face-vertices of $\aux(G)$.

By minimality of $S$ and the fact deficit is additive on components, $G_S$ is connected. Since each vertex of $G$ is on 3 faces and $G_S$ is the union of facial cycles, $G_S$ is 2-connected.

We examine the edges $\delta(G_S)$ from $G_S$ to the rest of $G$ and will eventually show it's a non-trivial 3-edge cut (contradicting cyclic 4-edge-connectivity).

Only vertices in $B_1$ or $W_1$ can have an incident edge in $\delta(G_S)$ since if $v$ lies on at least two faces of $S$, their facial cycles contain all three edges at $v$.

We now show all edges of $\delta(G_S)$ have an end in $B_1$.

\begin{claim}
$W_{S, 1}$ is empty
\end{claim}

\begin{proof}[Proof of the claim]
Suppose a white vertex $w \in W_1$ is on exactly one face $f \in S$. Then its neighbours in $\aux(G)$ going clockwise are

$f$, black in $G_S$, face not in $S$, black not in $G_S$, face not in $S$, black in $G_S$

The two black vertices in $G_S$ here have a neighbouring face not in $S$ so they're in $B_1$ or $B_2$ and removing $f$ increases their contribution to the deficit by 1 each so $d(S - f) \geq d(S)$, a contradiction to the minimality of $S$.
\end{proof}

So all edges of $\delta(G_S)$ have an end in $B_{S, 1}$ and the other (white) end outside $V(G_S)$ (since the other end can't be in any $W_{S, i}, i \geq 1$). By construction of $G_S$, every vertex of $B_{S, 1}$ is incident to exactly one edge in $\delta(G_S)$. So $|\delta(G_S)| = |B_{S, 1}|$.

Counting edges in $G[V(G_S)]$ from both the white and black vertices' side, we get

\[ 3|B \cap V(G_S)| - |\delta(G_S)| = 3|W \cap V(G_S)| \]

or

\[ |\delta(G_S)| = 3(|B \cap V(G_S)| - |W \cap V(G_S)|) \]

To bound this value, we apply Euler's formula to $G_S$. Let $k$ be the number of faces of $G_S$ that are not faces in $S$.

\begin{itemize}
\item $|V(G_S)| = |B_1| + |B_2| + |B_3| + |W_{S, 1}| + |W_2| + |W_3|$
\item $2|E(G_S)| = 2|B_1| + 3|B_2| + 3|B_3| + 2|W_{S, 1}| + 3|W_2| + 3|W_3|$
\item $|F(G_S)| = |S| + k$
\end{itemize}

\begin{align*}
|B_1| + |B_2| + |B_3| + |W_{S, 1}| + |W_2| + |W_3| + |S| + k &= 2 + |B_1| + (3/2)|B_2| + (3/2)|B_3| + |W_{S, 1}| + (3/2)|W_2| + (3/2)|W_3| \\
|S| + k &= 2 + (1/2)(|B_2| + |B_3| + |W_2| + |W_3|) \\
2|S| + 2k &= 4 + |B_2| + |B_3| + |W_2| + |W_3| \\
d(S) + 2k &= 4 - |B_1| - |B_2| - |B_3| + |W_2| + |W_3|
\end{align*}

Since $W_{S, 1}$ is empty, this is

\[ d(S) = 4 - 2k - (|B \cap V(G_S)| - |W \cap V(G_S)|) \]

and since $|\delta(G_S)| = 3(|B \cap V(G_S)| - |W \cap V(G_S)|)$

\[ d(S) = 4 - 2k - |\delta(G_S)|/3 \]

Since $d(S) \geq 1$, $k \leq 1$. $k = 1$ since $G_S$ has a face of $G$ not in $S$, namely the one containing $b_0$ in its interior. So $|\delta(G_S)| \leq 3$.

$\delta(G_S)$ is a non-trivial cut since one side contains $\{b_0, w_0\}$ and the other contains at least 4 vertices from any one facial cycle of a face in $S$. This is a contradiction to $G$ having no non-trivial 3-edge cuts.
\end{proof}

\section{Sign-reversing involutions}\label{sec:involutions}

In this section, we prove the Key Lemma, Lemma~\ref{lem:matching}. We first sketch the proof.

\begin{proof}[Proof sketch]
To prove the existence of a cycle-breaking perfect matching in $H = \GFB - N[b_0]$ for a rooted Barnette graph $(G, B, W, b_0w_0)$, we define two sign-reversing involutions on combinatorial objects based on these matchings. To do so, we define

\begin{enumerate}
\item some object $\cO$ and partition $\cO^+$, $\cO^-$ of $\cO$, and
\item two involutions $\iota_1$, $\iota_2: \cO \to \cO$ where non-fixed points map $\cO^+$ to objects in $\cO^-$ (and vice-versa).
\end{enumerate}

We construct $\iota_1$ so it fixes exactly the objects whose 2-factor is cycle-breaking. So if no cycle-breaking 2-factor exists, then $|\cO^+| = |\cO^-|$.

We construct $\iota_2$ so its set of fixed points is non-empty and they are all in the same part (say, $\cO^+$) so $|\cO^+| \neq |\cO^-|$.

So together they show that some cycle-breaking matching exists, proving the lemma.
\end{proof}

\subsection{Combinatorial objects}

We proceed with the definitions. Each 2-factor of $H$ is a union of cycles. Let $c_{\min}$ be the minimum number of cycles among all 2-factors of $H$.

\begin{definition}
A \emph{half-ray} of a white vertex $w$ is a straight upwards half-infinite line from $w$ (in our fixed drawing of $H$).
\end{definition}

\begin{definition}[of $\cO$]
Each object $o = (M_1, M_2, \{x_1, \ldots, x_{c_{\min}}\}) \in \cO$ is
\begin{enumerate}
\item an ordered 2-factor $(M_1, M_2)$ of $H$, and
\item $c_{\min}$ intersections between a half-ray of a white vertex $w$ and an edge of $M_2$.
\end{enumerate}
\end{definition}

Note the union of $M_1$ and $M_2$ is a 2-factor of $H$.

Equivalently, item 2 is $c_{\min}$ white vertices and $M_2$ edges $(w_i, e_{2,i})$ (for $i \in \{1, \ldots, c_{\min}\}$) where the half-ray of each $w_i$ intersects $e_{2,i}$. This second part helps us encode whether a white vertex is inside or outside a cycle in the fixed planar drawing of $H$ we chose.

For the remainder of this note, we restrict ourselves to the case where $c_{\min} = 1$ (i.e., $H$ is itself Hamiltonian) and extend this to higher values of $c_{\min}$ in Appendix~\ref{sec:large-cmin}. The extension is technical but straightforward.

Thus our objects are $(M_1, M_2, x)$ where $x$ is the intersection of a white vertex half-ray with an edge of $M_2$.

\subsection{Motivation}

Before defining $\cO^+$ and $\cO^-$ and the involutions, we note some basic properties of pairs of matchings we'll need. We call the \emph{sign} of an object + or - depending on its membership in $\cO^+$ or $\cO^-$.

Let $(M_1, M_2)$ be an ordered 2-factor of $H$.

Recall that $H$ is bipartite. Thus, for a cycle $C$ of $H[M_1 \cup M_2]$, we can swap the edges of $M_1$ and $M_2$ in $C$ to produce a new ordered 2-factor $(M_1 \symd C, M_2 \symd C)$ where $\symd$ is the symmetric difference operator defined as $X \symd Y = (X \cup Y) \setminus (X \cap Y)$. (We abuse notation and write $C$ here instead of $E(C)$.)

For the construction of our involutions, we often want $(M_1, M_2, x)$ and $(M_1 \symd C, M_2 \symd C, x)$ to have different signs. To achieve this, we define an orientation $\orient(M_1, M_2)$.

\begin{definition}
Given an ordered 2-factor $(M_1, M_2)$ of $H$, $\orient(M_1, M_2)$ is a directed graph on $V(H)$ where
\begin{itemize}
\item all edges of $M_1$ are directed from black vertex to face-vertex, and
\item all edges of $M_2$ are directed from face-vertex to black vertex.
\end{itemize}
\end{definition}

Since $M_1$ and $M_2$ are perfect matchings, $\orient(M_1, M_2)$ is thus a union of directed cycles.

For our fixed drawing of $H$, we can then define

$\cw'(M_1, M_2)$ as the number of clockwise cycles minus the number of counterclockwise cycles in $\orient(M_1, M_2)$

and

$\cw(M_1, M_2)$ as $\lfloor\cw'(M_1, M_2) / 2\rfloor \bmod 2$.

Note $\cw(M_1, M_2) \in \{0, 1\}$.

$\cw$ is chosen to have the following property.

\begin{lemma}
Let $(G, B, W, b_0w_0)$ be a rooted Barnette graph and $H = \GFB - N[b_0]$. For any ordered 2-factor $(M_1, M_2)$ of $H$ and any cycle $C$ of $H[M_1 \cup M_2]$, $\cw(M_1, M_2) \neq \cw(M_1 \symd C, M_2 \symd C)$.
\end{lemma}

\begin{proof}
Swapping $M_1$ and $M_2$ along $C$ changes the orientation of $C$ and so $\cw(M_1, M_2)$ and $\cw(M_1 \symd C, M_2 \symd C)$ have different values.
\end{proof}

We can also swap $M_1$ along any cycle $C$ of $H - M_2$, since it alternates between $M_1$-edges and edges of $E(H) - M_1 - M_2$. This changes $\cw'(M_1, M_2)$ by $\pm 2$, which we prove later.

\begin{lemma}\label{disk-lemma}
Let $(G, B, W, b_0w_0)$ be a rooted Barnette graph and $H = \GFB - N[b_0]$. For any ordered 2-factor $(M_1, M_2)$ of $H$ and any cycle $C$ of $H - M_2$, $\cw(M_1, M_2) \neq \cw(M_1 \symd C, M_2)$.
\end{lemma}

$\orient(M_1, M_2)$ also allows us to define a direction for edges crossing a white vertex half-ray.

\begin{definition}
For a white vertex $w$ and an edge $e_2$ of $M_2$ that crosses the half-ray of $w$, let $\dir(M_2, w, e_2) = 0$ if $e_2$ crosses the half-ray of $w$ from left to right in $\orient(M_1, M_2)$, and $\dir(M_2, w, e_2) = 1$ otherwise.
\end{definition}

Finally, we need the following facts about white vertex half-rays.

\begin{remark}
Let $C$ be a cycle of $\GFB$ and $w$ a white vertex. The half-ray of $w$ intersects $C$ an odd number of times if $w$ is in the interior of $C$ and an even number of times otherwise.
\end{remark}

We label these intersections and define a pairing of edges the half-ray of $w$ intersects.

\begin{definition}
For a set of edges $X$, if the half-ray of $w$ intersects $e_1$, ..., $e_m$ in that order starting at the point $w$ then $i$ is the \emph{position} of $e_i$.
\end{definition}

If the half-ray of $w$ intersects an even number of $X$ edges, we define a pairing of the intersected edges $p_X(w, e_i) = e_{i+1}$ and $p_X(w, e_{i+1}) = e_i$ for odd $i$. i.e., $p_X$ pairs $e_1 \leftrightarrow e_2$, $e_3 \leftrightarrow e_4$, ...

If the half-ray of $w$ intersects an odd number of $X$ edges, we define a pairing of the intersected edges $p_X(w, e_i) = e_{i \oplus 1}$ for $i \geq 2$. i.e., $p_X$ pairs $e_2 \leftrightarrow e_3$, $e_4 \leftrightarrow e_5$, ...

See Figure~\ref{fig:rays}.

\begin{figure}[t]
\centering
\begin{tikzpicture}[x=1cm,y=1cm]
\draw[m2A] (0.0,0.0) -- (2.4,0.0);
\draw[m1A] (2.4,0.0) -- (5.0,0.0);
\draw[m2A] (5.0,0.0) -- (5.0,2.0);
\draw[m1A] (5.0,2.0) -- (5.0,4.0);
\draw[m2A] (5.0,4.0) -- (2.2,4.0);
\draw[m1A] (2.2,4.0) -- (0.0,4.0);
\draw[m2A] (0.0,4.0) -- (0.0,3.2);
\draw[m1A] (0.0,3.2) -- (0.0,2.5);
\draw[m2A] (0.0,2.5) -- (2.0,2.5);
\draw[m1A] (2.0,2.5) -- (3.5,2.5);
\draw[m2A] (3.5,2.5) -- (3.5,1.5);
\draw[m1A] (3.5,1.5) -- (2.0,1.5);
\draw[m2A] (2.0,1.5) -- (0.0,1.5);
\draw[m1A] (0.0,1.5) -- (0.0,0.0);
\node[fv] at (0.0,0.0) {};
\node[bv] at (2.4,0.0) {};
\node[fv] at (5.0,0.0) {};
\node[bv] at (5.0,2.0) {};
\node[fv] at (5.0,4.0) {};
\node[bv] at (2.2,4.0) {};
\node[fv] at (0.0,4.0) {};
\node[bv] at (0.0,3.2) {};
\node[fv] at (0.0,2.5) {};
\node[bv] at (2.0,2.5) {};
\node[fv] at (3.5,2.5) {};
\node[bv] at (3.5,1.5) {};
\node[fv] at (2.0,1.5) {};
\node[bv] at (0.0,1.5) {};
\draw[rayE] (1.5,0.7) -- (1.5,4.8);
\node[wv] at (1.5,0.7) {};
\node[lbl,below=2pt] at (1.5,0.7) {$w$};
\node[xing] at (1.5,1.5) {};
\node[lbl,anchor=south west,inner sep=1pt] at (1.55,1.53) {$e_1$};
\node[xing] at (1.5,2.5) {};
\node[lbl,anchor=south west,inner sep=1pt] at (1.55,2.53) {$e_2$};
\node[xing] at (1.5,4.0) {};
\node[lbl,anchor=south west,inner sep=1pt] at (1.55,4.03) {$e_3$};
\draw[rayE] (2.5,2.0) -- (2.5,4.8);
\node[wv] at (2.5,2.0) {};
\node[lbl,below=2pt] at (2.5,2.0) {$w'$};
\node[xing] at (2.5,2.5) {};
\node[lbl,anchor=south west,inner sep=1pt] at (2.55,2.53) {$e_1$};
\node[xing] at (2.5,4.0) {};
\node[lbl,anchor=south west,inner sep=1pt] at (2.55,4.03) {$e_2$};
\node[note,anchor=west] at (5.5,3.4) {$w$ inside $C$: $e_1$ alone,};
\node[note,anchor=west] at (5.5,3.0) {$e_2 \leftrightarrow e_3$};
\node[note,anchor=west] at (5.5,2.2) {$w'$ outside $C$:};
\node[note,anchor=west] at (5.5,1.8) {$e_1 \leftrightarrow e_2$};
\end{tikzpicture}
\caption{$\orient(M_1, M_2)$ and half-rays from white vertices intersecting $M_1 \cup M_2$.}
\label{fig:rays}
\end{figure}
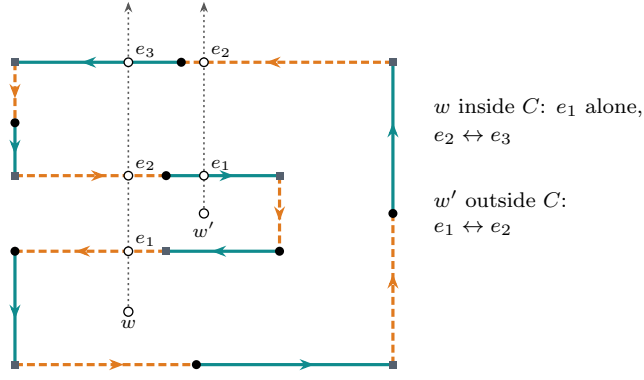

\subsection{Signed objects}

Now we define the sign $\sgn: \cO \to \mathbb Z_2$ of each object.

\begin{definition}
For an object $o = (M_1, M_2, (w_1, e_2)) \in \cO$,
\[ \sgn(M_1, M_2, (w_1, e_2)) = \cw(M_1, M_2) + \dir(M_2, w_1, e_2). \]
\end{definition}

Let $\cO^+ = \{o \mid o \in \cO, \sgn(o) = 0\}$, $\cO^- = \{o \mid o \in \cO, \sgn(o) = 1\}$.

To help break symmetry, fix a total order on the cycles of $\aux(G)$. The \emph{first cycle} with some property is the earliest cycle in this order with that property.

\subsection{Involutions}

The two involutions are defined as follows.

\begin{definition}
$\iota_1: \cO \to \cO$ is defined as
\begin{itemize}
\item $\iota_1(M_1, M_2, x) = (M_1, M_2, x)$ if $H - M_2$ contains no cycles
\item $\iota_1(M_1, M_2, x) = (M_1 \symd C, M_2, x)$ otherwise, where $C$ is the first cycle of $H - M_2$.
\end{itemize}
\end{definition}

$\iota_2$ is a bit more complex.

\begin{definition}
$\iota_2: \cO \to \cO$ is defined as follows.

For $o = (M_1, M_2, (w_1, e_2))$, let $C$ be the cycle of $H[M_1 \cup M_2]$ containing $e_2$.

$\iota_2$ is defined by the first of the four following conditions that applies.

\textbf{If some cycle of $H[M_1 \cup M_2]$ doesn't contain $e_2$}, $\iota_2(M_1, M_2, (w_1, e_2)) = (M_1 \symd C', M_2 \symd C', (w_1, e_2))$ where $C'$ is the first cycle of $H[M_1 \cup M_2]$ that doesn't contain $e_2$.

\textbf{If $w_1$ is not in the interior of $C$}, then
\begin{itemize}
\item $\iota_2(M_1, M_2, (w_1, e_2)) = (M_1, M_2, (w_1, p_C(w_1, e_2)))$ if $p_C(w_1, e_2) \in M_2$, and
\item $\iota_2(M_1, M_2, (w_1, e_2)) = (M_1 \symd C, M_2 \symd C, (w_1, p_C(w_1, e_2)))$ if $p_C(w_1, e_2) \in M_1$
\end{itemize}

\textbf{If $i_C(w_1, e_2) = 1$}, $\iota_2(M_1, M_2, (w_1, e_2)) = (M_1, M_2, (w_1, e_2))$

\textbf{If $i_C(w_1, e_2) > 1$} then
\begin{itemize}
\item $\iota_2(M_1, M_2, (w_1, e_2)) = (M_1, M_2, (w_1, p_C(w_1, e_2)))$ if $p_C(w_1, e_2) \in M_2$, and
\item $\iota_2(M_1, M_2, (w_1, e_2)) = (M_1 \symd C, M_2 \symd C, (w_1, p_C(w_1, e_2))) = (M_2, M_1, (w_1, p_C(w_1, e_2)))$ if $p_C(w_1, e_2) \in M_1$.
\end{itemize}
\end{definition}

We now need to prove that $\iota_1$ and $\iota_2$ are well defined and are indeed involutions.

\subsection{Properties of the involutions}

By construction, $\iota_1$ and $\iota_2$ are involutions, $\iota_1$ fixes exactly the objects $(M_1, M_2, x)$ where $M_2$ is cycle-breaking, and $\iota_2$ fixes exactly the objects $(M_1, M_2, (w_1, e_2))$ where $M_1 \cup M_2$ contains a single cycle $C$, $w_1$ is inside $C$ and $e_2$ is the first edge of $C$ the half-ray of $w_1$ meets.

It remains to check that $\iota_1$ and $\iota_2$ are sign-reversing and all objects $\iota_2$ fixes have the same sign. Checking $\iota_1$ and $\iota_2$ are sign-reversing is fairly routine and we leave it to Appendix~\ref{sign-reversing}. We now check $\iota_1$ is well defined and all objects $\iota_2$ fixes have the same sign.

\subsection{Proof of well definedness of $\iota_1$}

In the definition of $\iota_1$, we wrote $(M_1 \symd C, M_2, x)$ where $C$ is the first cycle of $H - M_2$. But we need to ensure $M_1 \symd C$ is still a perfect matching. We do this by showing $C$ alternates between $M_1$-edges and non-$M_1$-edges.

All cycles of $H - M_2$ alternate between black and face-vertices. Each black vertex $b$ has degree at most 2 in $H - M_2$ since $M_2$ is a perfect matching. Since $M_1$ is also a perfect matching, $b$ is incident to exactly one $M_1$-edge and therefore exactly one non-$M_1$-edge in $C$. Again since $M_1$ is a perfect matching from the face-vertex side, $C$ alternates between these edges.

\subsection{Fixed points of $\iota_2$}

We now prove this lemma.

\begin{lemma}\label{i2-fix}
Objects $\iota_2$ fixes all have the same sign.
\end{lemma}

\begin{proof}

$\iota_2$ fixes exactly the objects $(M_1, M_2, (w_1, e_2))$ where $M_1 \cup M_2$ contains a single cycle $C$, $w_1$ is inside $C$ and $e_2$ is the first edge of $C$ the half-ray of $w_1$ meets.

Since $w_1$ is inside $C$ and $e_2$ is the first edge of $C$ on the half-ray of $w_1$, $e_2$ crosses this half-ray from left to right if $C$ is clockwise in $\orient(M_1, M_2)$ and from right to left otherwise.

Thus in both cases either both $\cw$ and $\dir$ are flipped or neither is. So all such objects have the same sign.
\end{proof}

\subsection{Proof of Lemma~\ref{disk-lemma}}

We now prove Lemma~\ref{disk-lemma}. We'll split $\cw(M_1, M_2)$ into a sum of values on vertices so we can determine the change on a single cycle by looking at the change on vertices in that cycle.

To that end, for $b \in B - \{b_0\}$, let $\tau(M_1, M_2, b) = 1$ if on the face of $\GFB$ containing both the $M_1$-edge and $M_2$-edge incident to $b$, we rotate clockwise to get from the $M_2$-edge to the $M_1$-edge and 0 otherwise. Note such a face always exists since $b$ has degree 3.

Swapping along a cycle flips $\tau$ exactly at its black vertices. We must still control the number of 0 to 1 flips and 1 to 0 flips. We do this now.

Let $C$ be a cycle of $H$ that alternates between $M_1$-edges and non-$M_1$-edges. This holds for cycles of $H - M_2$ and for cycles of $H[M_1 \cup M_2]$. Each black vertex $b$ on $C$ has two edges of $\GFB$ on $C$ and one more, its \emph{third edge}. So let

\begin{itemize}
\item $t_{C, \mathrm{in}}$ be the number of black vertices on $C$ whose third edge is inside $C$, and
\item $t_{C, \mathrm{out}}$ be the number of black vertices on $C$ whose third edge is outside $C$.
\end{itemize}

For a cycle $C$ of $H - M_2$ the third edge is the $M_2$-edge. For a cycle $C$ of $H[M_1 \cup M_2]$ it is the edge to the third face, which may be a root face (and is then outside $C$).

\begin{claim}\label{tin-tout}
$t_{C, \mathrm{in}} = |C|/4 - 3/2$ and $t_{C, \mathrm{out}} = |C|/4 + 3/2$
\end{claim}

\begin{proof}
Let $H_C$ be the subgraph of $\GFB$ induced by all vertices and edges inside $\Int(C)$, including its boundary. We will apply Euler's formula to $H_C$.

Let $\beta$ be the number of black vertices in $V(H_C)$. That is, $\beta = |B \cap V(H_C)|$

All vertices inside $C$ are in $H$, since $b_0$ is on the outer face and each root face is joined to $b_0$ by an edge that cannot cross $C$. Since $M_1$ is a perfect matching of $H$, $M_1$ matches all vertices in $H_C$ to each other (since no edge of $M_1$ crosses $C$ and it already matches all vertices of $C$ by choice of $C$). So $|V(H_C)| = 2 \beta$.

Every edge of $\GFB$ has exactly one black end. Since $G$ is cubic, we count all edges of $H_C$ from the black vertex side, giving $|E(H_C)| = 3 \beta - t_{C, \mathrm{out}}$.

$H_C$ has one outer face of length $|C|$ and a number of inner faces, each of length 6 (each face of $\GFB$ surrounds one white vertex $w$ and is bounded by the 3 black neighbours of $w$ and the 3 faces at $w$). As each edge is on two faces, $|F(H_C)| = 1 + (2|E(H_C)| - |C|) / 6$

By Euler's formula,

\begin{align*}
|V(H_C)| + |F(H_C)| - |E(H_C)| &= 2 \\
2 \beta + 1 + |E(H_C)| / 3 - |C|/6 - |E(H_C)| &= 2 \\
2 \beta - |C|/6 - (2/3)|E(H_C)| &= 1 \\
2 \beta - |C|/6 - (2/3)(3 \beta - t_{C, \mathrm{out}}) &= 1 \\
- |C|/6 + (2/3) t_{C, \mathrm{out}} &= 1 \\
- |C|/4 + t_{C, \mathrm{out}} &= 3/2 \\
t_{C, \mathrm{out}} &= |C|/4 + 3/2
\end{align*}

Since $t_{C, \mathrm{in}} + t_{C, \mathrm{out}} = |B \cap C| = |C|/2$, we have $t_{C, \mathrm{in}} = |C|/4 - 3/2$
\end{proof}

We now show $\cw'$ is the sum of all taus up to a constant depending only on $H$ and a multiplicative factor.

\begin{claim}
\[ \cw'(M_1, M_2) = |V(H)|/6 - (2/3)\sum_b \tau(M_1, M_2, b) \]
\end{claim}

\begin{proof}
Let $D$ be a cycle of $\orient(M_1, M_2)$ and $\beta_D = |B \cap D|$. At each black vertex $b$ of $D$, $D$ enters along the $M_2$-edge of $b$ and leaves along the $M_1$-edge of $b$. The four possible values of $\tau$ are thus.

\begin{center}
\begin{tabular}{c|cc}
$\tau$ & third edge outside $D$ & third edge inside $D$ \\
\hline
clockwise & 0 & 1 \\
counterclockwise & 1 & 0 \\
\end{tabular}
\end{center}

So $\sum_{b \in D} \tau(M_1, M_2, b)$ is $t_{D, \mathrm{in}}$ if $D$ is clockwise and $t_{D, \mathrm{out}}$ if $D$ is counterclockwise.

By the Claim~\ref{tin-tout}, $\sum_{b \in D} \tau(M_1, M_2, b)$ is $|D|/4 - 3/2$ if $D$ is clockwise and $|D|/4 + 3/2$ otherwise.

Each black vertex is on exactly one cycle of $\orient(M_1, M_2)$ so summing over all cycles

\[ \sum_b \tau(M_1, M_2, b) = |V(H)|/4 - (3/2) \#\cw + (3/2) \#\mathrm{ccw} = |V(H)|/4 - (3/2) \cw'(M_1, M_2) \]

where $\#\cw$ and $\#\mathrm{ccw}$ are the number of clockwise and counterclockwise cycles of $\orient(M_1, M_2)$
\end{proof}

\begin{proof}[Proof of Lemma~\ref{disk-lemma}]
Swapping $M_1$ along $C$ flips $\tau$ on the $|C|/2$ black vertices of $C$ and nowhere else. The total change is $\pm(t_{C, \mathrm{in}} - t_{C, \mathrm{out}}) = \pm3$. So $\cw'$ changes by $-(2/3)(\pm3) = \pm2$. So $\cw$ flips.
\end{proof}

\subsection{Completing the proof of the Key Lemma}

\begin{proof}[Proof of Lemma~\ref{lem:matching}]
By Lemmas~\ref{i1-sign},~\ref{i2-sign}, $\iota_1$ and $\iota_2$ are sign-reversing involutions of $\cO$. By Lemma~\ref{i2-fix}, objects $\iota_2$ fixes all have the same sign. By Lemma~\ref{two-factor}, there is a 2-factor of $H$ and therefore an ordered 2-factor $(M_1, M_2)$ of $H$ that is the union of $c_{\min}$ cycles.

Recall $c_{\min} = 1$ (with $c_{\min} > 1$ handled in Appendix~\ref{sec:large-cmin}). Let $w_1$ be any white vertex inside the only cycle of $M_1 \cup M_2$, let $e_2$ be the first edge of $M_1 \cup M_2$ the half-ray of $w_1$ intersects. We may assume $e_2 \in M_2$ (otherwise swap the labels $M_1$ and $M_2$). Let $o = (M_1, M_2, (w_1, e_2))$

Then by definition, $\iota_2$ fixes $o$ so the objects $\iota_2$ fixes are non-empty. So $|\cO^+| \neq |\cO^-|$ as $\iota_2$ is sign-reversing, matching opposite signed objects. So $\iota_1$ has a fixed point $(M_1', M_2', x')$, and $(M_1', M_2')$ is a cycle-breaking 2-factor.
\end{proof}

\section{From 2-factor to spanning trees}\label{sec:spanning-tree}

Having proven the Key Lemma, it remains to show that the 2-factor of $G$ we build from it is indeed Hamiltonian. Recall from the proof sketch that we do this in three steps.

\begin{enumerate}
\item Find a cycle-breaking 2-factor $(M_1, M_2)$ and take the dual $M_2^*$ of $M_2$ and see that it is a spanning tree $T_W$ of white vertices in $\GFB^*$.
\item ``Trace around'' $T_W$ to get a spanning tree of $G - b_0$.
\item Remove all outgoing edges from white vertices along with $e_0$ to get a Hamiltonian cycle of $G$.
\end{enumerate}

In this section, we prove Lemma~\ref{lem:spanning-tree} which is the last property we need for the first step.

\begin{restatement}[Lemma]{lem:spanning-tree}
Let $(G, B, W, b_0w_0)$ be a rooted Barnette graph, $H = \GFB - N[b_0]$ and $(M_1, M_2)$ a cycle-breaking 2-factor of $H$.

Then $\GFB - M_2$ is a spanning tree of $\GFB$.
\end{restatement}

\begin{proof}
We now show $\GFB - M_2$ is acyclic. Then by counting vertices and edges of $\GFB - M_2$, we conclude $\GFB - M_2$ is a spanning tree.

Suppose for a contradiction $\GFB - M_2$ has a (minimal) cycle $C$. Since $H - M_2$ is acyclic and $\GFB[N[b_0]]$ is acyclic, the part of $C$ in $H$ is a path $P_1$ between two black vertices that are neighbours of a (possibly the same) root vertex. Let $P$ be the subpath of $P_1$ where we remove the two root vertices at the end.

$P$ is contained entirely in $H$, starts and ends at a black vertex and alternates black and face-vertices. So $P$ has one more black vertex than face-vertex.

Since black vertices have degree 3 in $G$, they have degree 1 or 2 in $H - M_2$ depending on whether they're adjacent to a root vertex in $\GFB$. So all (face) neighbours of black vertices of $P$ are also on $P$. But $P$ has one more black vertex than face-vertex so $M_1$ can't match all black vertices on $P$. A contradiction to $M_1$ being a matching.
\end{proof}

\begin{corollary}\label{cor:spanning-tree}
Let $(G, B, W, b_0w_0)$ be a rooted Barnette graph, $H = \GFB - N[b_0]$ and $(M_1, M_2)$ a cycle-breaking 2-factor of $H$.

Then $M_2^*$ is a spanning tree of white vertices.
\end{corollary}

\begin{proof}
Since $\GFB$ is planar, the planar dual of the complement of $T_{FB}$ is a spanning tree $T_W$ of white vertices by tree-cotree decomposition.
\end{proof}

\section{From white spanning tree to Hamiltonian cycle in $G$}\label{sec:ham}

Let $T_W$ be the dual complement of $T_{FB}$. It is a spanning tree of $\GFB^*$.

For this section, we work in both $\aux(G)$ and $G^*$.

\begin{definition}
The \emph{black dual trace} of an edge $b_1w_1 \in E(\aux(G))$, $b_1 \in B$, $w_1 \in W$ is the set of two edges $\{b_1f_1, b_1f_2\}$ where
\begin{itemize}
\item $f_1$ is the next clockwise neighbour of $b_1$ after $w_1$, and
\item $f_2$ is the next counterclockwise neighbour of $b_1$ after $w_1$.
\end{itemize}

The black dual trace of a set of edges is the union of the black dual trace of each edge.
\end{definition}

A black dual trace can contain an edge twice, but we only trace matchings, which never produce duplicates.

The black dual trace of a set of edges between $B$ and $F(G)$ in $\aux(G)$ is defined similarly, exchanging the roles of $W$ and $F(G)$.

\begin{definition}
Call neighbours of $b_0$ that are face-vertices \emph{root face-vertices}.
\end{definition}

Let $f_{\mathrm{out}}$ be the root face-vertex representing the outer face of $G$ in $\aux(G)$.

\begin{lemma}\label{lem:ham}
Let $(G, B, W, b_0w_0)$ be a rooted Barnette graph, $H = \GFB - N[b_0]$ and $M_2$ a perfect matching of $H$ where $H - M_2$ is acyclic and $T_W = M_2^*$ is a spanning tree of white vertices.

Let $T_{\vec W}$ be the oriented spanning tree of $G - b_0$ obtained by taking the black dual trace of $T_W$ and orienting all edges towards $w_0$.

Let $P$ be the set of all outgoing edges from white vertices in $T_{\vec W}$ together with $e_0$.

Then $E(G) - P$ induces a Hamiltonian cycle of $G$.
\end{lemma}

\begin{proof}

\begin{claim}
$P$ is a perfect matching of $G$.
\end{claim}

\begin{proof}
Each edge of $T_W$ is the dual of an $M_2$-edge between a black $b$ and a face $f$ and is traced as a path wbw' through $b$. Since $M_2$ matches each black $b \neq b_0$ exactly once, $b$ lies on exactly one such path, and exactly one half of it leaves a white vertex in the orientation $T_{\vec W}$. Each white $w \neq w_0$ has exactly one outgoing edge, and $e_0$ covers $b_0$ and $w_0$. So $P$ is a perfect matching of $G$.
\end{proof}

Since $G$ is cubic, $G - P$ is a 2-factor. Let $\Pbar = E(G) - P$. We abuse notation and also refer to $\Pbar$ as the subgraph of $\aux(G)$ induced by $E(G) - P$. It remains to show $\Pbar$ contains a single cycle.

The cycles of the 2-factor $\Pbar$ and the dual complement $P^*$ don't cross in the planar drawing as their edges are disjoint. So $\Pbar$ and the black dual trace $\trB(P)$ of $P$ also do not cross, though they may intersect at black vertices. That is, the cycles of $\Pbar$ separate $\aux(G)$ into regions (one exterior of all cycles and the interior of each cycle with the interior of cycles within it removed) and every component of $P^*$ is entirely contained in one region.

Suppose for a contradiction that $\Pbar$ contains (at least) two cycles.

\begin{claim}
Some cycle $C$ of $\Pbar$ contains no root face-vertex in $\Int(C)$.
\end{claim}

\begin{proof}
Since $e_0$ is in the outer face of $\aux(G) - f_{\mathrm{out}}$ and $e_0$ is in $P$, $\Ext(\Pbar)$ contains both ends of $e_0^*$ ($f_{\mathrm{out}}$ and another root face-vertex). Let $f_3$ be the third root face-vertex, so $b_0f_3$ is the edge with no endpoint in $e_0^*$. So only $f_3$ can be in the interior of a cycle of $\Pbar$. $f_3$ can only be contained in the interior of one cycle of $\Pbar$ as $b_0f_3$ is a single edge and cannot cross the boundary of two different cycles. As $\Pbar$ has at least two cycles, some cycle $C$ of $\Pbar$ has no root face-vertex in $\Int(C)$.
\end{proof}

Let $C_1$ be such a cycle with no other cycle of $\Pbar$ inside $\Int(C_1)$.

Define an orientation $\vec \trB(P)$ of $\trB(P)$ as follows. Orient every $M_2$-edge from the face-vertex to black vertex. Orient $f_{\mathrm{out}}$ to $b_0$. Orient all remaining edges from the black vertex to face-vertex.

For each black $b \neq b_0$, the edge of $T_W$ dual to its $M_2$-edge, between $b$ and a face $f$, joins the two white neighbours of $b$ on either side of $f$. The $P$-edge at $b$ is one half of that path, so $f$ is next to it around $b$, and the $M_2$-edge is in $\trB(P)$. So $M_2 \subseteq \trB(P)$.

\begin{claim}
Every vertex not in $F(G) \cap N(b_0)$ has outdegree 1 in $\vec \trB(P)$
\end{claim}

\begin{proof}
Indeed, since $M_2$ is a perfect matching of $H$, non-root face-vertices have outdegree 1. Since $P$ is a perfect matching, each black vertex is incident to exactly one edge of $P$ and since this is the dual trace of $P$, all black vertices have outdegree 1.
\end{proof}

Take a face-vertex $f_1$ inside $C_1$ (it exists as $C_1$ is a black-white cycle and thus contains a facial cycle). Starting at $f_1$, repeatedly follow outgoing edges. We either end at a root face-vertex or in a directed cycle of $\vec \trB(P)$.

Our path and $C_1$ don't cross since our path is a subset of $\trB(P)$ and $C_1$ is in $\Pbar$. Since we start in $\Int(C_1)$, the cycle or root face-vertex we end at is also in $\Int(C_1)$.

Since $C_1$ contains no root face-vertex, our path ends at a directed cycle $C_2$. But $C_2$ is a cycle of $\trB(P)$ and contains a white vertex in its interior (any face of $\aux(G)[B \cup F(G)]$). Contradiction as $\Pbar$ spans all white vertices and $C_1$ contains no other cycles and therefore no other white vertices inside its interior.

Therefore $\Pbar$ is a single cycle and hence a Hamiltonian cycle.
\end{proof}

Now the proof of Theorem~\ref{cyclic4} is a routine combination of previous lemmas.

\begin{proof}[Proof of Theorem~\ref{cyclic4}]
By Lemma~\ref{lem:matching}, $\GFB - N[b_0]$ has a cycle-breaking 2-factor $(M_1, M_2)$. By Corollary~\ref{cor:spanning-tree}, $M_2^*$ as dual edges of $\GFB$ is a spanning tree of white vertices of $\GFB^*$.

Construct $P$ from Lemma~\ref{lem:ham}. That is, let $T_{\vec W}$ be the oriented spanning tree of $G - b_0$ obtained by taking the black dual trace of $T_W$ and orienting all edges towards $w_0$. Let $P$ be the set of all outgoing edges from white vertices in $T_{\vec W}$ together with $e_0$.

By Lemma~\ref{lem:ham}, $E(G) - P$ induces a Hamiltonian cycle of $G$ avoiding $e_0 \in P$.
\end{proof}

\textbf{Statement of AI use} LLMs were used to read and understand the OpenAI preprint, to revise and format this note and to generate figures from their description.

\appendix
\section{Sign reversal}\label{sign-reversing}

We now prove $\iota_1$ and $\iota_2$ are sign-reversing

\begin{lemma}\label{i1-sign}
For all $o \in \cO$, $\iota_1(o) = o$ or $\sgn(\iota_1(o)) \neq \sgn(o)$
\end{lemma}

\begin{proof}
$\iota_1(M_1, M_2, x)$ is fixed if $H - M_2$ contains no cycles and

$\iota_1(M_1, M_2, x) = (M_1 \symd C, M_2, x)$, where $C$ is the first cycle of $H - M_2$.

Note that $x = (w_1, e_2)$ is unaffected as $M_2$ is unchanged and its orientation in $\orient(M_1, M_2)$ is unchanged ($\orient(M_1, M_2)$ orients the edges of $M_1$ and $M_2$ independently). So $\dir(M_2, w_1, e_2)$ is also unchanged.

By Lemma~\ref{disk-lemma}, $\cw(M_1, M_2) \neq \cw(M_1 \symd C, M_2)$ and so $(M_1, M_2, x)$ and $(M_1 \symd C, M_2, x)$ have opposite signs.
\end{proof}

\begin{lemma}\label{i2-sign}
For all $o \in \cO$, $\iota_2(o) = o$ or $\sgn(\iota_2(o)) \neq \sgn(o)$
\end{lemma}

\begin{proof}
We consider each case in the definition of $\iota_2$.

Let $C$ be the cycle of $H[M_1 \cup M_2]$ containing $e_2$.

\textbf{If some cycle of $H[M_1 \cup M_2]$ doesn't contain $e_2$}, $\iota_2(M_1, M_2, (w_1, e_2)) = (M_1 \symd C', M_2 \symd C', (w_1, e_2))$ where $C'$ is the first cycle of $H[M_1 \cup M_2]$ that doesn't contain $e_2$.

Since $w_1$ and $C$ are unchanged, the orientation of $e_2$ in $\orient(M_1, M_2)$ and $\orient(M_1 \symd C', M_2 \symd C')$ is the same and thus the half-ray intersection direction $\dir$ remains unchanged.

Since $M_1$ and $M_2$ swap edges along $C'$, the orientation of $C'$ is reversed between $\orient(M_1, M_2)$ and $\orient(M_1 \symd C', M_2 \symd C')$. So $\cw(M_1, M_2)$ and $\cw(M_1 \symd C', M_2 \symd C')$ have different parities.

Otherwise, all cycles of $H[M_1 \cup M_2]$ contain $e_2$, so $H[M_1 \cup M_2]$ is a single cycle.

\textbf{If $w_1$ is not in the interior of $C$}, then
\begin{itemize}
\item $\iota_2(M_1, M_2, (w_1, e_2)) = (M_1, M_2, (w_1, p_C(w_1, e_2)))$ if $p_C(w_1, e_2) \in M_2$, and
\item $\iota_2(M_1, M_2, (w_1, e_2)) = (M_1 \symd C, M_2 \symd C, (w_1, p_C(w_1, e_2)))$ if $p_C(w_1, e_2) \in M_1$
\end{itemize}

In the first case, $M_1$ and $M_2$ remain unchanged so the $\cw$ is unchanged. $p_C$ pairs consecutive edges so $e_2$ and $p_C(w_1, e_2)$ cross the half-ray of $w_1$ in opposite directions. So $\dir(M_2, w_1, e_2) \neq \dir(M_2, w_1, p_C(w_1, e_2))$. So $\iota_2$ is sign-reversing in this case, as required.

In the second case, the direction of consecutive edges the half-ray of $w_1$ meets still changes direction but $M_1$ and $M_2$ swap on $C$ so the orientation of all cycles, including $C$, changes so $\dir(M_2, w_1, e_2) = \dir(M_2 \symd C, w_1, p_C(w_1, e_2))$ as there are two orientation changes. But since we swap on $C$, $\cw$ changes as well. So $\iota_2$ is sign-reversing in this case, as required.

Otherwise $M_1 \cup M_2$ is a single (Hamiltonian) cycle of $H$ and $w_1$ is in $\Int(C)$.

\textbf{If $i_C(w_1, e_2) = 1$}, $\iota_2(M_1, M_2, (w_1, e_2)) = (M_1, M_2, (w_1, e_2))$

This is a fixed point so nothing to prove.

\textbf{If $i_C(w_1, e_2) > 1$} then
\begin{itemize}
\item $\iota_2(M_1, M_2, (w_1, e_2)) = (M_1, M_2, (w_1, p_C(w_1, e_2)))$ if $p_C(w_1, e_2) \in M_2$, and
\item $\iota_2(M_1, M_2, (w_1, e_2)) = (M_1 \symd C, M_2 \symd C, (w_1, p_C(w_1, e_2))) = (M_2, M_1, (w_1, p_C(w_1, e_2)))$ if $p_C(w_1, e_2) \in M_1$.
\end{itemize}

This is the same as the case where $w_1$ is outside $C$. $e_2$ and $p_C(w_1, e_2)$ are consecutive edges the half-ray of $w_1$ intersects. In the first case, $\dir$ changes but $\cw$ remains the same while in the second case, $\cw$ changes but $\dir$ remains the same, having been flipped twice.
\end{proof}

\section{Extending to 2-factors with more cycles}\label{sec:large-cmin}

Recall $c_{\min}$ is the minimum number of cycles among all 2-factors of $H$. We defined and verified Lemma~\ref{lem:matching} for $c_{\min} = 1$. This is when $H$ has a 2-factor containing a single (Hamiltonian) cycle.

When $c_{\min} > 1$, our combinatorial objects have $c_{\min}$ intersections of $M_2$-edges and white half-rays along with the usual two matchings.

Let $o = (M_1, M_2, \{(w_1, e_{2,1}), \ldots, (w_{c_{\min}}, e_{2,c_{\min}})\}) \in \cO$ be such an object.

\[ \sgn(M_1, M_2, (w_1, e_{2,1}), \ldots, (w_{c_{\min}}, e_{2,c_{\min}})) = \cw(M_1, M_2) + \sum_{i=1}^{c_{\min}} \dir(M_2, w_i, e_{2,i}) \]

$\iota_1$ remains unchanged as its condition depends only on one of the matchings $M_2$ and not the intersections.

Most cases of $\iota_2$ also remain unchanged. Namely when there is a cycle without any intersection edges.

The only remaining case is when there are exactly $c_{\min}$ cycles in $H[M_1 \cup M_2]$ and each cycle $C_i$ contains a different edge $e_{2,i} \in \{e_{2,1}, \ldots, e_{2,c_{\min}}\}$.

If some $w_i$ is not in $\Int(C_i)$ for some $i$, apply the single cycle involution to the first cycle $C_i$ that doesn't contain $w_i$ among all $i$.

\begin{itemize}
\item $\iota_2(M_1, M_2, (w_1, e_{2,1}), \ldots, (w_i, e_{2,i}), \ldots, (w_{c_{\min}}, e_{2,c_{\min}}))$ = 
        $(M_1, M_2, (w_1, e_{2,1}), \ldots, (w_i, p_{C_i}(w_i, e_{2,i})), \ldots, (w_{c_{\min}}, e_{2,c_{\min}}))$ if $p_{C_i}(w_i, e_{2,i}) \in M_2$, and
\item $\iota_2(M_1, M_2, (w_1, e_{2,1}), \ldots, (w_i, e_{2,i}), \ldots, (w_{c_{\min}}, e_{2,c_{\min}}))$ = 
        $(M_1 \symd C_i, M_2 \symd C_i, (w_1, e_{2,1}), \ldots, (w_i, p_{C_i}(w_i, e_{2,i})), \ldots, (w_{c_{\min}}, e_{2,c_{\min}}))$ if $p_{C_i}(w_i, e_{2,i}) \in M_1$
\end{itemize}

In the remaining case, each $w_i$ is in $\Int(C_i)$.

Otherwise, switch to the other edge $p_{C_i}(w_i, e_{2,i})$ along the first cycle $C_i$ where $e_{2,i}$ isn't the first edge on $w_i's$ half-ray.

So $\iota_2$ keeps fixed objects with exactly $c_{\min}$ cycles with intersecting first edges and corresponding white vertex inside each cycle. Each cycle contributes the same value to the sign so the overall sign of any fixed object is $\lfloor c_{\min}/2\rfloor \bmod 2$.

\section{3-edge cuts}\label{sec:3edgescut}

In this appendix, we complete the proof of the main theorem, which we now restate, by handling the case where $G$ has a non-trivial 3-edge cut.

\begin{restatement}[Theorem]{stronger}
For every Barnette graph $G$ and every edge $e_0 \in G$, $G - e_0$ has a Hamiltonian cycle.
\end{restatement}

\begin{proof}
Let $G$ be a minimum counterexample to Theorem~\ref{stronger}. If $G$ has no non-trivial 3-edge cut, by Theorem~\ref{cyclic4}, $G$ is not a counterexample.

Let $\delta(X)$ be a non-trivial 3-edge cut and $\delta(X) = \delta_B \cup \delta_W$ where

\begin{itemize}
\item $\delta_B$ is all edges of $\delta(X)$ with a black end in $X$, and
\item $\delta_W$ is all edges of $\delta(X)$ with a white end in $X$.
\end{itemize}

Counting edges in $G[X]$ from the white and black sides gives

\[ 3|B \cap X| - |\delta_B| = 3|W \cap X| - |\delta_W| \]

rearranging to

\[ 3(|B \cap X| - |W \cap X|) = |\delta_B| - |\delta_W| \]

Since the RHS is in $\{-3, -1, 1, 3\}$ and the LHS is a multiple of 3, the RHS is -3 or 3. In other words, all edges in $\delta(X)$ have the same coloured end in $X$, say black. If two edges of $\delta(X)$ shared an end $v \in X$, then $\delta(X - v)$ would be a 2-edge cut. So all edges of $\delta(X)$ have different ends in $X$, and by symmetry different ends in $V(G) - X$. Hence $G/X$ and $G/(V(G) - X)$ are simple. If $G[X]$ were disconnected, some component $K$ would have $|\delta(K)| \leq 1$, a bridge. So $G[X]$ is connected, and by symmetry so is $G[V(G) - X]$.

\begin{claim}
$G_1 = G/X$, the graph obtained from $G$ by contracting all of $X$ to a single vertex $x$, is a Barnette graph.
\end{claim}

\begin{proof}[Proof of the claim]
Since all ends of $\delta(X)$ in $X$ have the same colour, $G/X$ is bipartite. It is cubic since $|\delta(X)| = 3$ and all other vertex degrees are the same in $G$ and $G_1$. Contracting the connected subgraph $G[X]$ preserves planarity so $G_1$ is planar.

Since $G_1$ is simple and cubic, a $\leq2$ vertex cut $Y$ is a $\leq2$ edge cut by taking the appropriate edges in $\delta(Y)$. Since an edge cut of $G_1$ is an edge cut of the same size in $G$, $G_1$ is 3-connected.
\end{proof}

By symmetry, $G_2 = G/(V(G) - X)$ is a Barnette graph. By symmetry, we may assume $e_0$ does not have both ends in $X$ and is therefore in $E(G_1)$. Both $G_1$ and $G_2$ are smaller than $G$ since both sides of a non-trivial cut have at least 2 vertices. By minimality $G_1$ has a Hamiltonian cycle $C_1$ missing $e_0$. Say it also doesn't use $e' \in \delta(X)$.

By minimality, $G_2$ has a Hamiltonian cycle $C_2$ missing $e'$. Now $C_1 \cup C_2$ is a Hamiltonian cycle of $G$, a contradiction.
\end{proof}

\section*{Table of symbols}

{\small
\begin{longtable}{@{}p{0.20\textwidth}|p{0.76\textwidth}@{}}
Symbol & Typical meaning \\
\hline
\endhead
$(G, B, W, b_0w_0)$ & a rooted Barnette graph \\
$B, W$ & black and white vertices of $G$ \\
$B_{S, i}, W_{S, i}, B_i, W_i$ & black (white) vertices with exactly $i$ neighbouring faces in $S$ \\
$c_{\min}$ & minimum number of cycles of a 2-factor of $H$ \\
$d(S)$ & deficit of a set $S$ \\
$f_{\mathrm{out}}$ & the face-vertex of the outer face of $G$ \\
$G_S$ & graph induced by the union of the facial cycles of the faces in $S$ \\
$\GFB$ & face-black vertex incidence graph of $G$ \\
$H$ & $\GFB - N[b_0]$ \\
$p_X(w, e_i)$ & partner of $e_i$ in the pairing along the half-ray of $w$ \\
$(M_1, M_2)$ & an ordered 2-factor \\
$T_{BW}$ & black dual trace of $T_W$, a subgraph of $G - b_0$ \\
$T_{FB}$ & the spanning tree $\GFB - M_2$ of $\GFB$ \\
$T_W$ & the spanning tree $M_2^*$ of the white vertices \\
$T_{\vec W}$ & $T_{BW}$ oriented towards $w_0$ \\
$t_{C, \mathrm{in}}, t_{C, \mathrm{out}}$ & black vertices of $C$ whose third edge is inside (outside) $C$ \\
$\aux(G)$ & $G$ plus a face-vertex in each face, joined to all vertices of the face \\
$\cw'(M_1, M_2)$ & clockwise minus counterclockwise cycles of $\orient(M_1, M_2)$ \\
$\cw(M_1, M_2)$ & $\lfloor\cw'(M_1, M_2)/2\rfloor \bmod 2$ \\
$\dir(M_2, w, e_2)$ & 0 if $e_2$ crosses the half-ray of $w$ from left to right, 1 otherwise \\
$\Ext(C), \Int(C)$ & exterior and interior of the cycle $C$, including $C$ \\
$\orient(M_1, M_2)$ & orientation of $M_1 \cup M_2$ with $M_1$ directed black to face-vertex, $M_2$ directed face-vertex to black \\
$\sgn(o)$ & $\cw + \dir \bmod 2$ \\
$\trB(X)$ & black dual trace of a set of edges $X$ \\
$\delta(X)$ & cut edges: edges with exactly one end in $X$ \\
$\delta_B, \delta_W$ & edges of $\delta(X)$ with a black (white) end in $X$ \\
$\iota_1, \iota_2$ & two sign-reversing involutions on $\cO$ \\
$\tau(M_1, M_2, b)$ & 1 if the turn at $b$ from its $M_2$-edge to its $M_1$-edge is clockwise \\
$\cO$ & the set of objects $(M_1, M_2, (w_1, e_2))$ \\
$\cO^+, \cO^-$ & objects of sign 0 and of sign 1 \\
$X^*$ & planar dual of a graph, or the dual edges of a set of edges \\
$\Pbar$ & $E(G) - P$ \\
$X \symd Y$ & $(X \cup Y) - (X \cap Y)$ \\
$x \oplus 1$ & $x$ with its lowest binary bit flipped \\
\end{longtable}
}

\end{document}